\documentclass[12pt]{article}
\usepackage{amsmath,amsfonts,amsthm,amssymb}
\usepackage[a4paper,margin=3cm]{geometry}
\usepackage{hyperref}

\newtheorem{thm}{Theorem}[section]%
\newtheorem{lem}[thm]{Lemma}%
\theoremstyle{bodyrm}%

\newtheorem{rem}[thm]{Remark}%
\newcommand{\dC}{\mathbb{C}}  
\newcommand{\dE}{\mathbb{E}}

 \newcommand{\dP}{\mathbb{P}} 
 \newcommand{\dR}{\mathbb{R}}

\newcommand{\bA}{\mathbf{A}} \newcommand{\bB}{\mathbf{B}}
\newcommand{\bC}{\mathbf{C}}

\newcommand{\bI}{\mathbf{I}} 
 
\newcommand{\bM}{\mathbf{M}}

 \newcommand{\bX}{\mathbf{X}}

\newcommand{\al}{\alpha}      
\newcommand{\be}{\beta}
\newcommand{\De}{\Delta}
\newcommand{\de}{\delta}

\newcommand{\la}{\lambda}

\newcommand{\veps}{\varepsilon}

\newcommand{\ABS}[1]{{{\left| #1 \right|}}} % |1|
\newcommand{\NRM}[1]{{{\left\| #1\right\|}}} % ||1||
\newcommand{\PAR}[1]{{{\left(#1\right)}}} % (1)
\renewcommand{\leq}{\leqslant}
\renewcommand{\geq}{\geqslant}

\title{Circular law for non-central random matrices} %
\author{Djalil~\textsc{Chafa\"\i}} %
\date{\small Preprint, June 2008. Revised March 2010.}

\begin{document}

\maketitle

\begin{abstract}
  Let $(X_{jk})_{j,k\geq 1}$ be an infinite array of i.i.d.\ complex random
  variables, with mean $0$ and variance $1$. Let $\la_{n,1},\ldots,\la_{n,n}$
  be the eigenvalues of $(\frac{1}{\sqrt{n}}X_{jk})_{1\leq j,k\leq n}$. The
  strong circular law theorem states that with probability one, the empirical
  spectral distribution $\frac{1}{n}(\de_{\la_{n,1}}+\cdots+\de_{\la_{n,n}})$
  converges weakly as $n\to\infty$ to the uniform law over the unit disc
  $\{z\in\dC;|z|\leq1\}$. In this short note, we provide an elementary
  argument that allows to add a deterministic matrix $M$ to $(X_{jk})_{1\leq
    j,k\leq n}$ provided that $\mathrm{Tr}(MM^*)=O(n^2)$ and
  $\mathrm{rank}(M)=O(n^\al)$ with $\al<1$. Conveniently, the argument is
  similar to the one used for the non--central version of Wigner's and
  Marchenko-Pastur theorems.
\end{abstract}

\bigskip

{\footnotesize \textbf{AMS 2010 Mathematical Subject Classification:} 15B52.}


\bigskip

\section{Introduction}

For any square $n\times n$ matrix $\bA$ with complex entries, let the complex
eigenvalues $\la_1(\bA),\ldots,\la_n(\bA)$ of $\bA$ be labeled so that
$\ABS{\la_1(\bA)}\geq\cdots\geq\ABS{\la_n(\bA)}$. The \emph{empirical spectral
  distribution} of $\bA$ is the discrete probability measure
$\mu_{\bA}:=\frac{1}{n}\sum_{k=1}^n\de_{\la_k(\bA)}$. We denote by
$s_1(\bA)\geq\cdots\geq s_n(\bA)$ the \emph{singular values} of $\bA$, i.e.\
the eigenvalues of the positive semi-definite Hermitian matrix
$\sqrt{\bA\bA^*}$ where $\bA^*$ is the conjugate--transpose of $\bA$. The
\emph{operator norm} is $s_1(\bA)=\max_{\NRM{x}_2=1}\NRM{\bA x}_2$ and the
square \emph{Hilbert-Schmidt} norm is
$\NRM{\bA}^2:=s_1(\bA)^2+\cdots+s_n(\bA)^2=\mathrm{Tr}(\bA\bA^*)=\sum_{j,k=1}^n|\bA_{j,k}|^2$.
Weyl's inequality $|\la_1(\bA)|^2+\cdots+|\la_n(\bA)|^2\leq
s_1(\bA)^2+\cdots+s_n(\bA)^2$ ensures that the second moment of $\mu_{\bA}$ is
always bounded above by $\frac{1}{n}\NRM{\bA}^2$. The following result was
recently obtained by Tao and Vu \cite[Corollary 1.15]{tao-vu-cirlaw-bis}.

\begin{thm}[Circular law for central random matrices]\label{th:nccl}
  Let $(X_{jk})_{j,k\geq1}$ be i.i.d.\ complex random variables. Let
  $(M_{jk})_{j,k\geq1}$ be deterministic complex numbers. For every integer
  $n\geq1$, set $\bX_n=(X_{jk})_{1\leq j,k\leq n}$ and $\bM_n=(M_{jk})_{1\leq
    j,k\leq n}$. If
  \begin{itemize}
  \item $\dE[|X_{1,1}|^2]=1$ and $\dE[X_{1,1}]=0$
  \item $\NRM{\bM_n}^2=O(n^2)$ and $\mathrm{rank}(\bM_n)=O(n^\al)$
    for some $\al<1$
  \end{itemize}
  then with probability one, $\mu_{\frac{1}{\sqrt{n}}(\bX_n+\bM_n)}$ tends
  weakly as $n\to\infty$ to the uniform distribution on the unit disc
  $\{z\in\dC;\ABS{z}\leq1\}$ (known as the circular law).
\end{thm}

The aim of this note is to provide an alternative and elementary argument
which reduces theorem \ref{th:nccl} to the central case where $\bM_n\equiv0$
for every $n$. Conveniently, the approach is close in spirit to the one used by Bai
\cite{MR1711663} for the derivation of Wigner's and Marchenko-Pastur theorems
for non--central random matrices. 

This note was motivated by the study of random Markov matrices, including the
Dirichlet Markov Ensemble \cite{chafai,MR2549497}, for which a circular law
theorem is conjectured. The initial version of this note was written before
the apparition of \cite{tao-vu-cirlaw-bis}, and provided for the first time a
non--central version of the circular law theorem. The initial version was
based on potential theoretic tools. For convenience, the present version makes
use instead of the replacement principle borrowed from
\cite{tao-vu-cirlaw-bis}.

Theorem \ref{th:nccl} belongs to a sequence of works by many authors,
including Mehta \cite{MR0220494}, Girko \cite{MR773436}, Silverstein
\cite{MR841088}, Bai \cite{MR1428519}, Edelman \cite{MR1437734} \'Sniady
\cite{MR1929504}, Bai and Silverstein \cite{bai-silverstein-book}, Pan and
Zhou \cite{pan-zhou}, G\"otze and Tikhomirov \cite{gotze-tikhomirov}, and Tao
and Vu \cite{tao-vu-circular-law}.

\begin{rem}[Constant case] Consider the case where the entries of $\bM_n$ are
  all equal to $1$ in theorem \ref{th:nccl}. We have then
  $\mathrm{rank}(\bM_n)=1$ and $s_1(\bM_n)=n$. Suppose additionally that
  $X_{1,1}$ has finite fourth moment. Then, by Bai and Yin theorem
  \cite{MR1235416}, with probability one,
  $\lim_{n\to\infty}s_1(\frac{1}{\sqrt{n}}\bX_n)=2$, and thus
  $\frac{1}{\sqrt{n}}(\bX_n+\bM_n)$ is a random bounded perturbation of the
  rank one symmetric matrix $\frac{1}{\sqrt{n}}\bM_n$ which has spectrum
  $$
  \la_n(\bM_n)=\cdots=\la_2(\bM_n)=0
  \quad\text{and}\quad
  \la_1(\bM_n)=\sqrt{n}.
  $$
  From this observation, Silverstein \cite{MR1284550} has shown, via
  perturbation techniques such as Bauer-Fike and Gerschgorin theorems, that
  with probability one,
  $$
  \ABS{\la_2\PAR{\frac{1}{\sqrt{n}}(\bX_n+\bM_n)}}\leq 2+o(1)
  \quad\text{and}\quad %
  \ABS{\la_1\PAR{\frac{1}{\sqrt{n}}(\bX_n+\bM_n)}-\sqrt{n}} \leq 2+o(1).
  $$
  See also the work of Andrew \cite{MR1062321}. Also, with probability one, as
  $n\to\infty$, the spectral radius $|\la_1(\frac{1}{\sqrt{n}}(\bX_n+\bM_n))|$
  blows up while $\mu_{\frac{1}{\sqrt{n}}(\bX_n+\bM_n)}$ remains weakly
  localized.
\end{rem}

\section{Reduction to the central case}

In order to show that theorem \ref{th:nccl} reduces to the central case where
$\bM_n\equiv0$ for every $n\geq1$, it suffices to check the assumptions of the
replacement principle of theorem \ref{th:replacement} with
$\bA_n:=\frac{1}{\sqrt{n}}\bX_n$ and $\bB_n:=\frac{1}{\sqrt{n}}(\bX_n+\bM_n)$.
By the strong law of large numbers and the assumption on $\NRM{\bM_n}$, with
probability one,
$$
\frac{1}{n}\NRM{\bA_n}^2+\frac{1}{n}\NRM{\bB_n}^2=O(1).
$$
Next, by theorem \eqref{th:tv-sn} and the first Borel-Cantelli lemma, for all
$z\in\dC$, with probability one, the random matrices $\bA_n-z\bI_n$ and
$\bB_n-z\bI_n$ are invertible for large enough $n$. Let us define, for large
enough $n$, the quantity
$$
\Delta_{n,z}
:=\frac{1}{n}\log\ABS{\det\PAR{\bA_n-z\bI_n}}
-\frac{1}{n}\log\ABS{\det\PAR{\bB_n-z\bI_n}}.
$$
If we set $\mu_{n,z}:=\mu_{\sqrt{(\bA_n-z\bI_n)(\bA_n-z\bI_n)^*}}$ and
$\nu_{n,z}:=\mu_{\sqrt{(\bB_n-z\bI_n)(\bB_n-z\bI_n)^*}}$ then
$$
\Delta_{n,z}=\int_0^\infty\!\log(t)\,d(\mu_{n,z}-\nu_{n,z})(t).
$$
By the strong law of large numbers and the assumption on $\NRM{\bM_n}$, for
all $z\in\dC$, with probability one, there exists $a>0$ such that
$$
\max(s_1(\bA_n-z\bI_n),s_1(\bB_n-z\bI_n))\leq n^a
$$
for large enough $n$. On the other hand, by theorem \eqref{th:tv-sn} and the
first Borel-Cantelli lemma, for all $z\in\dC$, with probability one, there
exists $b>0$ such that
$$
\min(s_n(\bA_n-z\bI_n),s_n(\bB_n-z\bI_n))\geq n^{-b}
$$
for large enough $n$. Therefore, with $\al_n:=n^{-b}$ and $\be_n:=n^a$,
and large enough $n$,
$$
\De_{n,z}=\int_{\al_n}^{\be_n}\!\log(t)\,d(\mu_{n,z}-\nu_{n,z})(t).
$$
Let $F_{n,z}$ and $G_{n,z}$ be the cumulative distribution functions of the
real probability measures $\mu_{n,z}$ and $\nu_{n,z}$. By lemma
\ref{le:rank} and the assumption on $\mathrm{rank}(\bM_n)$, for almost all
$z\in\dC$, with probability one, there exists $\veps>0$ such that
$$
\NRM{F_{n,z}-G_{n,z}}_\infty = O(n^{-\veps}).
$$
Therefore, by lemma \ref{le:IBP}, we obtain, for almost all $z\in\dC$, with
probability one,
$$
\ABS{\De_{n,z}} \leq (\log(\be_n)-\log(\al_n))\NRM{F_{n,z}-G_{n,z}}_\infty
=o(1).
$$

\section{Tools}\label{se:tools}

This section gathers some tools used in our proof of theorem \ref{th:nccl}. By
Green's theorem, for any complex polynomial $P$ and smooth compactly supported
$f:\dC\to\dR$,
$$
\int_{\dC}\!f\,d\mu=\frac{1}{2\pi}\int_{\dC}\!\De f\log|P|\,dxdy
$$
where $\mu:=\de_{\la_1}+\cdots+\de_{\la_n}$ is the counting measure of the
roots $\la_1,\ldots,\la_n$ of $P$ in $\dC$. Used for characteristic
polynomials of random matrices, this identity provides, via dominated
convergence arguments, the following theorem, see \cite[Theorem
2.1]{tao-vu-cirlaw-bis}.

\begin{thm}[Replacement principle]\label{th:replacement}
  Let $(\bA_n)_{n\geq1}$ and $(\bB_n)_{n\geq1}$ be two sequences of complex
  random matrices where $\bA_n,\bB_n$ are $n\times n$, without any
  assumptions. If
  \begin{itemize}
  \item with probability one
    $\frac{1}{n}\NRM{\bA_n}^2+\frac{1}{n}\NRM{\bB_n}^2=O(1)$
  \item for almost all $z\in\dC$, with probability one, the random matrices
    $\bA_n-z\bI_n$ and $\bB_n-z\bI_n$ are invertible for large enough $n$
  \item for almost all $z\in\dC$, with probability one,
    $$
    \lim_{n\to\infty}\PAR{\frac{1}{n}\log\ABS{\det\PAR{\bA_n-z\bI_n}}-\frac{1}{n}\log\ABS{\det\PAR{\bB_n-z\bI_n}}}=0
    $$
  \end{itemize}
  then with probability one, $\mu_{\bA_n}-\mu_{\bB_n}$ tends weakly to zero as
  $n\to\infty$.
\end{thm}

The following lemma is a special case of the integration by parts formula for
the Lebesgue-Stieltjes integral (with atoms). We give a short proof for
convenience.

\begin{lem}[Integration by parts]\label{le:IBP}
  If $a_1,\ldots,a_n,b_1,\ldots,b_n\in[\al,\be]\subset\dR$, and $F_\mu$ and
  $F_\nu$ are the cumulative distribution functions of
  $\mu=\frac{1}{n}(\de_{a_1}+\cdots+\de_{a_n})$ and
  $\nu=\frac{1}{n}(\de_{b_1}+\cdots+\de_{b_n})$ respectively, then for any
  smooth $f:[\al,\be]\to\dR$,
  $$
  \int_\al^\be\!f(x)\,d\mu(x)
  -
  \int_\al^\be\!f(x)\,d\nu(x)
  =\int_\al^\be\!f'(x)(F_\mu(x)-F_\nu(x))\,dx.
  $$
  In particular, when $f$ is non decreasing,
  $$
  \ABS{\int_\al^\be\!f(x)\,d\mu(x)-\int_\al^\be\!f(x)\,d\nu(x)}
  \leq (f(\be)-f(\al))\NRM{F_\mu-F_\nu}_\infty.
  $$
\end{lem}

\begin{proof}
  One can assume by continuity that $a_1,\ldots,a_n,b_1,\ldots,b_n$ are all
  different. We reorder $a_1,\ldots,a_n,b_1,\ldots,b_n$ into
  $c_1\leq\cdots\leq c_{2n}$. For every $1\leq k\leq 2n$, set $\veps_k=+1$ if
  $c_k\in\{a_1,\ldots,a_n\}$ and $\veps_k=-1$ if $c_k\in\{b_1,\ldots,b_n\}$.
  We have
  $$
  \int_\al^\be\!f(x)\,d\mu(x)
  -
  \int_\al^\be\!f(x)\,d\nu(x)
  =
  \frac{1}{n}\sum_{k=1}^n(f(a_i)-f(b_i))
  =\frac{1}{n}\sum_{k=1}^{2n}\veps_k f(c_k).
  $$
  By an Abel transform, we get by denoting $S_k=\veps_1+\cdots+\veps_k$,
  $$
  \sum_{k=1}^{2n}\veps_k f(c_k)
  =-\sum_{k=1}^{2n-1} S_k(f(c_{k+1})-f(c_k))+S_{2n}f(c_{2n}).
  $$
  Since $F_\mu-F_\nu$ is constant and equal to $S_k$ on $[c_k,c_{k+1}[$,
  $$
  S_k(f(c_{k+1})-f(c_k))
  =\int_{c_k}^{c_{k+1}}\!f'(x)(F_\mu(x)-F_\nu(x))\,dx.
  $$
  It remains to notice that $S_{2n}=F_\mu(c_{2n})-F_\nu(c_{2n})=0$.
\end{proof}

The following lemma is a direct consequence of interlacing inequalities for
singular values obtained by Thompson \cite{MR0407051} in 1976. It was also
obtained by Bai \cite{MR1711663} and generalized by Benaych-Georges and Rao
\cite{benaych-georges-rao}. It is worthwhile to mention that it gives neither
an upper bound for $s_1(\bB),\ldots,s_k(\bB)$ nor a lower bound for
$s_{n-k+1}(\bB),\ldots,s_n(\bB)$ where $k:=\mathrm{rank}(\bA-\bB)$, even in
the case $k=1$.

\begin{lem}[Rank inequality]\label{le:rank}
  Let $\bA$ and $\bB$ be two $n\times m$ complex matrices. Let
  $F_{\sqrt{\bA\bA^*}},F_{\sqrt{\bB\bB^*}}$ be the cumulative distribution
  functions of $\mu_{\sqrt{\bA\bA^*}}$ and $\mu_{\sqrt{\bB\bB^*}}$. Then
  $$
  \NRM{F_{\sqrt{\bA\bA^*}}-F_{\sqrt{\bB\bB^*}}}_\infty %
  \leq \frac{1}{n}\mathrm{rank}(\bA-\bB).
  $$
\end{lem}

The following theorem is due to Tao and Vu \cite[Theorem
2.1]{tao-vu-circular-law}, and is inspired from the work of Rudelson and
Vershynin \cite{MR2407948}.

\begin{thm}[Polynomial bounds for smallest singular values]
  \label{th:tv-sn}
  Let $L$ be a probability distribution on $\dC$ with finite and non--zero
  variance. For every constants $A>0$ and $C_1>0$, there exists constants
  $B>0$ and $C_2>0$ such that for every $n\times n$ random matrix $\bX$ with
  i.i.d.\ entries of law $L$ and every $n\times n$ deterministic matrix $\bC$
  with $s_1(\bC)\leq n^{C_1}$, we have
  $$
  \dP(s_n(\bX+\bC)\leq n^{-B})\leq C_2 n^{-A}.
  $$
\end{thm}

\bigskip

\textbf{Acknowledgments.} The final form of this note benefited from the
comments of anonymous referees. Part of this work was done during two visits
to the \textsc{Laboratoire Jean Dieudonn\'e} in Nice, France. The author would
like to thank Pierre \textsc{Del Moral} and Persi \textsc{Diaconis} for their
kind hospitality there, and also Neil \textsc{O'Connell} for his
encouragements.

\providecommand{\bysame}{\leavevmode\hbox to3em{\hrulefill}\thinspace}
\providecommand{\MR}{\relax\ifhmode\unskip\space\fi MR }
\providecommand{\MRhref}[2]{%
  \href{http://www.ams.org/mathscinet-getitem?mr=#1}{#2}
}
\providecommand{\href}[2]{#2}

\bigskip

\vfill\thispagestyle{empty}

{\footnotesize
  \noindent Djalil~\textsc{Chafa\"\i} \qquad
  \textbf{E-mail:} \texttt{djalil(at)chafai.net} \\
  \textsc{Laboratoire d'Analyse et de Math\'ematiques Appliqu\'ees}\\
  \textsc{UMR 8050 CNRS Universit\'e Paris-Est Marne-la-Vall\'ee} \\
  \textsc{5 Boulevard Descartes, F-77454 Cedex 2, Champs-sur-Marne, France.} }

\end{document}